\documentclass[11pt]{article}
\usepackage[margin=1.05in]{geometry}
\usepackage{amsmath,amssymb,amsthm}
\usepackage[T1]{fontenc}
\usepackage{lmodern}
\usepackage{microtype}
\usepackage[hidelinks]{hyperref}

\newtheorem{theorem}{Theorem}[section]
\newtheorem{proposition}[theorem]{Proposition}
\newtheorem{lemma}[theorem]{Lemma}
\newtheorem{remark}[theorem]{Remark}

\newcommand{\Pp}{P}

\title{Convergence of the critical probabilities for stopping near the top of a random walk}
\author{Jos\'e A. Islas\\
Facultad de Ciencias Naturales y Exactas (FACINE), Universidad Aut\'onoma de Sinaloa,\\
Culiac\'an, Sinaloa, Mexico}
\date{}

\begin{document}
\maketitle

\begin{abstract}
Let $\{S_n\}_{n\ge0}$ be a Bernoulli random walk with $P(X_i=1)=p$ and $P(X_i=-1)=q=1-p$, and let $M_n=\max\{S_0,\ldots,S_n\}$. We consider the semi-best selection problem of maximizing $P(M_N-S_\tau\le1)$ over stopping rules $\tau\le N$. Allaart showed that the optimal decision in the critical state $(n,1)$ is determined by critical probabilities $p_n$, and conjectured that $p_n\to1/2$. In this note, we prove this conjecture. For every fixed $p<1/2$, the stopping value converges to $\nu=1-(p/q)^2$, while the continuation value converges to $2p\nu$. Since $2p\nu<\nu$, stopping is optimal in state $(n,1)$ for all sufficiently large $n$. Together with the known bound $p_n<1/2$ for $n\ge4$, this yields $\lim_{n\to\infty}p_n=1/2$.
\end{abstract}


\section{Introduction}

Consider $\{S_n\}_{n=0,1,\ldots}$ to be a Bernoulli random walk with parameter $p\in(0,1)$. Thus $S_0:=0$ and, for $n\ge1$,
\[
S_n:=X_1+\cdots+X_n,
\]
where $X_1,X_2,\ldots$ are independent, identically distributed random variables with
\[
\Pp(X_1=1)=p,\qquad \Pp(X_1=-1)=q:=1-p.
\]
Let
\[
M_n:=\max\{S_0,S_1,\ldots,S_n\}.
\]

For a given finite time horizon $N$, the problem of maximizing the probability of stopping at the top,
\[
\Pp(S_\tau=M_N),
\]
was examined by Hlynka and Sheahan \cite{HlynkaSheahan1988} for $p=1/2$. Yam, Yung and Zhou \cite{YamYungZhou2009} later considered the general Bernoulli case.

Allaart \cite{Allaart2010} extended the problem to maximize the probability of stopping at the top or one unit from the top, the semi-best selection problem. The objective is to find a stopping rule $\tau$ maximizing
\begin{equation}\label{eq:objective}
\Pp(M_N-S_\tau\le1).
\end{equation}

Following \cite{Allaart2010,Islas2024}, put
\[
Z_n:=M_n-S_n.
\]
If $k=N-n$ steps remain and $Z_n=j$, denote the state by $(k,j)$. Stopping in $(k,j)$ gives probability zero when $j\ge2$, and gives
\[
U_k:=\Pp(M_k\le1)
\]
when $j=0$ or $1$. Thus the decision of interest is the comparison in the critical states $(k,1)$, as in \cite{Allaart2010}. No further structural result from \cite{Allaart2010} is used below beyond Proposition \ref{prop:allaart-threshold} and the bound \eqref{eq:upper}.

Let $V_n$ be the optimal win probability from state $(n,1)$, and let $W_n$ be the optimal win probability from $(n,1)$ if one continues at least one more step. Then
\[
V_n=\max\{U_n,W_n\}.
\]
As in \cite{Islas2024}, let $\pi_{n,i}$ denote the optimal win probability from state $(n,i)$ when a state-specific value is needed. Thus $V_n=\pi_{n,1}$.

For convenience, we collect the main results of Allaart \cite{Allaart2010} in the notation used in \cite{Islas2024}.

\begin{proposition}[Allaart]\label{prop:allaart-threshold}
There exists critical values $p_n$ such that, in state $(n,1)$, it is optimal to stop if and only if $p\le p_n$.
\end{proposition}
At $p=p_n$, stopping is optimal, possibly tied with continuation. This tie does not affect the limiting argument below.

\begin{lemma}[Allaart]\label{lem:allaart-basic}
The critical values satisfy
\[
p_1=1,\qquad p_2=p_3=\frac12,
\]
\begin{equation}\label{eq:upper}
p_n<\frac12,\qquad n\ge4,
\end{equation}
and
\[
\limsup_{n\to\infty}p_n=\frac12.
\]
\end{lemma}
The statement $\limsup_{n\to\infty}p_n=1/2$ is also due to Allaart but is not used below; for the upper bound in the main theorem, \eqref{eq:upper} already gives $\limsup_{n\to\infty}p_n\le1/2$.

Allaart also proved the interlacing inequality
\[
p_{2m+1}\ge p_{2m-1}\ge p_{2m},\qquad m\ge4.
\]
He further conjectured that
\[
p_{2m}\le p_{2m+2},\qquad m\ge2,
\]
and, more importantly for the present paper,
\begin{equation}\label{eq:conjecture}
\lim_{n\to\infty}p_n=\frac12.
\end{equation}
In \cite{Islas2024}, further oscillation properties and inequalities for $p_n$ were obtained, but the convergence conjecture remained open. The purpose of this paper is to prove \eqref{eq:conjecture} by a short long-horizon comparison. The proof uses from Allaart only Proposition \ref{prop:allaart-threshold} and the bound \eqref{eq:upper}.

\section{Long-horizon win probabilities}

Fix
\[
0<p<\frac12,\qquad q=1-p,\qquad \rho:=\frac pq<1.
\]
We use the standard gambler's ruin fact that, for $a\ge1$, the probability that the walk ever reaches level $a$ is
\begin{equation}\label{eq:hitting}
\rho^a.
\end{equation}

The events $\{M_n\le1\}$ decrease to the event that the walk never reaches level $2$. Hence, by \eqref{eq:hitting},
\begin{equation}\label{eq:Ulimit}
U_n=\Pp(M_n\le1)\longrightarrow
\nu:=1-\rho^2>0.
\end{equation}

The following proposition gives the limits needed in the proof of the main result.

\begin{proposition}\label{prop:limits}
For fixed $p<1/2$,
\begin{equation}\label{eq:neededlimits}
\pi_{n,0}\longrightarrow\nu,
\qquad
\pi_{n,2}\longrightarrow\rho\nu.
\end{equation}
\end{proposition}

\begin{proof}
We first consider state $(n,0)$. Stopping immediately gives
\[
\pi_{n,0}\ge U_n,
\]
so by \eqref{eq:Ulimit},
\[
\liminf_{n\to\infty}\pi_{n,0}\ge\nu.
\]
Write
\[
U_m=\nu+\varepsilon_m,
\qquad \varepsilon_m\downarrow0.
\]
We shall use the following elementary observation. Suppose the process starts from state $(n,i)$, and let $X'_1,X'_2,\ldots$ be the subsequent increments, with
\[
S'_m:=X'_1+\cdots+X'_m,
\qquad
M'_m:=\max\{0,S'_1,\ldots,S'_m\}.
\]
After $m$ further steps the process is in a state $(n-m,Z'_m)$, where
\[
Z'_m=\max\{i,M'_m\}-S'_m\ge i-S'_m.
\]
Consequently,
\begin{equation}\label{eq:lastL}
\Pp(Z'_m\le1)\le \Pp(S'_m\ge i-1)=\Pp(S_m\ge i-1),
\end{equation}
since $S'_m$ has the same distribution as $S_m$.
Since $p<1/2$, the random walk tends to move downward. Thus, for every fixed $i$,
\begin{equation}\label{eq:downward}
\Pp(S_m\ge i-1)\longrightarrow0.
\end{equation}

Fix $L\ge1$, let $n>L$, and consider any stopping rule starting from $(n,0)$. If it stops with at least $L$ steps remaining, its conditional win probability is at most
\[
U_L=\nu+\varepsilon_L,
\]
because $U_m$ is nonincreasing in $m$. If it stops with fewer than $L$ steps remaining and wins, then $Z'_m\le1$ for some $m$ among the last $L$ steps. Hence, by \eqref{eq:lastL},
\[
\pi_{n,0}
\le U_L+
\sum_{m=n-L+1}^{n}\Pp(S_m\ge-1).
\]
Here the sum runs over the $L$ times $m=n-L+1,\ldots,n$ at which fewer than $L$ steps remain; each term tends to zero as $n\to\infty$ by \eqref{eq:downward}, and there are only $L$ terms. Therefore
\[
\limsup_{n\to\infty}\pi_{n,0}\le U_L.
\]
Letting $L\to\infty$ and using \eqref{eq:Ulimit} gives
\[
\pi_{n,0}\longrightarrow\nu.
\]

Now consider state $(n,2)$. A winning stop is possible only after the process reaches a state $(k,1)$. Before this first visit, from any state $(k,j)$ with $j\ge2$ the next state is $(k-1,j-1)$ with probability $p$ and $(k-1,j+1)$ with probability $q$. By the gambler's ruin formula, the probability of ever reaching a state $(k,1)$ starting from $(n,2)$ is
\[
\rho=\frac pq.
\]
For the process started from state $(n,2)$, let
\[
T_1:=\inf\{m\ge0:Z'_m=1\}
\]
be the first number of steps required to reach a state of the form $(k,1)$.

For an upper bound, fix $L\ge1$, let $n>L$, and consider any stopping rule from $(n,2)$. If it stops with at least $L$ steps remaining and wins, the process must first reach a state $(k,1)$, an event of probability at most $\rho$. At the stopping state, at least $L$ steps remain, so the conditional win probability is at most $U_L$. Therefore this contribution is at most $\rho U_L$. If it stops with fewer than $L$ steps remaining and wins, then $Z'_m\le1$ for some $m$ among the last $L$ steps. Hence, by \eqref{eq:lastL},
\[
\pi_{n,2}
\le \rho U_L+
\sum_{m=n-L+1}^{n}\Pp(S_m\ge1).
\]
Using \eqref{eq:downward},
\[
\limsup_{n\to\infty}\pi_{n,2}\le\rho U_L.
\]
Letting $L\to\infty$ gives
\[
\limsup_{n\to\infty}\pi_{n,2}\le\rho\nu.
\]

For the lower bound, fix $K\ge1$ and use the rule that stops at the first visit to a state $(k,1)$ if $T_1\le K$. For $n>K$,
\[
\pi_{n,2}\ge
\sum_{k=1}^{K}\Pp(T_1=k)U_{n-k}.
\]
Letting $n\to\infty$ gives
\[
\liminf_{n\to\infty}\pi_{n,2}
\ge \nu\Pp(T_1\le K).
\]
Finally, letting $K\to\infty$ and using
\[
\Pp(T_1<\infty)=\rho,
\]
we obtain
\[
\liminf_{n\to\infty}\pi_{n,2}\ge\rho\nu.
\]
Hence
\[
\pi_{n,2}\longrightarrow\rho\nu.
\]
\end{proof}

\section{The critical probabilities}

We can now prove the conjecture stated by Allaart and considered further in \cite{Islas2024}.

\begin{theorem}\label{thm:main}
The critical probabilities satisfy
\[
\lim_{n\to\infty}p_n=\frac12.
\]
\end{theorem}

\begin{proof}
Fix any $p<1/2$. From the critical state $(n,1)$, if one continues for one step, the next state is $(n-1,0)$ with probability $p$ and $(n-1,2)$ with probability $q$. Hence
\begin{equation}\label{eq:W}
W_n=p\pi_{n-1,0}+q\pi_{n-1,2}.
\end{equation}
By Proposition \ref{prop:limits},
\[
W_n\longrightarrow p\nu+q\rho\nu
=p\nu+q\frac pq\nu
=2p\nu.
\]
By \eqref{eq:Ulimit}, $U_n\to\nu$. Therefore
\begin{equation}\label{eq:difference}
U_n-W_n\longrightarrow \nu(1-2p)>0.
\end{equation}
Thus $U_n>W_n$ for all sufficiently large $n$. Hence, for every fixed $p<1/2$, it is eventually optimal to stop in state $(n,1)$. By the definition of $p_n$,
\[
p\le p_n
\]
for all sufficiently large $n$. Since this holds for every $p<1/2$,
\[
\liminf_{n\to\infty}p_n\ge\frac12.
\]
Together with Allaart's bound \eqref{eq:upper},
\[
\limsup_{n\to\infty}p_n\le\frac12,
\]
and the result follows.
\end{proof}

\begin{remark}
The proof uses from Allaart only Proposition \ref{prop:allaart-threshold} and the bound $p_n<1/2$ for $n\ge4$. It does not require monotonicity of $p_n$ or separate arguments for the even and odd subsequences. For each fixed $p<1/2$,
\[
U_n\to\nu,
\qquad
W_n\to2p\nu,
\]
so $V_n=U_n$ for all sufficiently large $n$, and consequently $V_n\to\nu$.
\end{remark}

\begin{remark}
The index from which stopping is optimal depends on the fixed value of $p$ and is not uniform as $p\uparrow1/2$. The argument proves convergence of $p_n$ but does not provide a rate.
\end{remark}

\section*{Declarations}

\noindent\textbf{Funding.} The author received no specific funding for this work.

\noindent\textbf{Competing interests.} The author declares no competing interests.

\noindent\textbf{Use of generative AI.} During the preparation of this manuscript, the author used OpenAI's ChatGPT to assist with proof checking, literature checking, and language editing. The author reviewed and verified the mathematical arguments and references and takes full responsibility for the content of the manuscript.

\end{document}